\documentclass[fontsize=10pt,twoside ]{scrreprt}

\usepackage[square,numbers,nonamebreak]{natbib}
\usepackage{setspace}

\usepackage[english]{babel}

\usepackage{chngcntr}

\usepackage[hang]{footmisc}
 at 11truept
\font\ssc=pplrc9d at 11 truept
\newcommand\qedbox{$\rlap{$\sqcap$}\sqcup$}

\usepackage[    protrusion=true,
            expansion=true,
            final,
kerning=true,
spacing=true,
factor=1100,stretch=10,shrink=10,
            babel,
	tracking=alltext,
letterspace = 10
                ]{microtype}
\usepackage[hmarginratio=1:1, bottom=2.5cm, top=2.5cm, inner=2cm, outer=2cm, paperwidth=17cm, paperheight=24cm]{geometry}

\usepackage[automark]{scrlayer-scrpage}
\cohead{\textnormal{Permutation modules for Ramsey structures}}
\lehead{\pagemark}
\rohead{\pagemark}

\let\ceheadL\cehead
\renewcommand\cehead[1]{
\ceheadL{\textnormal{#1}}
}

\cehead{David M. Evans}
\lefoot{}
\rofoot{}

\usepackage{graphicx}
\usepackage[usenames,dvipsnames,svgnames,table]{xcolor}
\definecolor{Maroon}{cmyk}{0, 0.87, 0.68, 0.32}
\definecolor{RoyalBlue2}{cmyk}{80,100,0,0.1}

\newcommand\auths[1]{\large \textsc{\textcolor{Maroon}{#1}}\setstretch{1.2}}
\newcommand\titl[1]{\center \linespread{1.1}\color{RoyalBlue2}\Large\textbf{ #1}\color{black}\bigskip} 
\renewcommand\abstract[1]{
\begin{center}
{\textbf{Abstract}}
\end{center}
{
\linespread{1.1}\fontsize{9pt}{-10pt}\selectfont #1}}

\usepackage[utf8]{inputenc}
\usepackage[sc,osf]{mathpazo}
\usepackage[euler-digits]{eulervm}
\usepackage[T1]{fontenc}
\usepackage{mathtools}

\DeclareSymbolFont{operators}{\encodingdefault}{ppl}{m}{n}
\DeclareMathAlphabet{\mathbf}{\encodingdefault}{ppl}{bx}{n}
\DeclareMathAlphabet{\mathit}{\encodingdefault}{ppl}{m}{it}

\usepackage{amsfonts}
\usepackage{amsmath}
\usepackage{ragged2e}
\usepackage{titlesec}

\renewcommand{\thesection}{\arabic{section}}
 
\titleformat{\section}{\medskip\bigskip\normalfont\Large\bf}{\thesection}{0.5em}{}
\titleformat{\subsection}{\smallskip\bigskip\normalfont\large\bf}{\thesubsection}{0.5em}{}
\titleformat{\subsubsection}{\smallskip\bigskip\normalfont\large\bf}{\thesubsubsection}{0.5em}{}

\usepackage{amsthm}
\newtheoremstyle{dotless}{}{}{\itshape}{}{\bfseries}{}{1em}{}

\theoremstyle{dotless}
\newtheorem{theorem}{Theorem}[section]
\newtheorem{lemma}[theorem]{Lemma}
\newtheorem{proposition}[theorem]{Proposition}
\newtheorem{corollary}[theorem]{Corollary}

\newtheorem{definition}[theorem]{Definition}
\newtheorem{remarks}[theorem]{Remarks}

\newtheorem{question}[theorem]{Question}

\counterwithout{theorem}{section}
\numberwithin{theorem}{section}

\DeclareOldFontCommand{\rm}{\normalfont\rmfamily}{\mathrm}
\DeclareOldFontCommand{\sf}{\normalfont\sffamily}{\mathsf}
\DeclareOldFontCommand{\tt}{\normalfont\ttfamily}{\mathtt}
\DeclareOldFontCommand{\bf}{\normalfont\bfseries}{\mathbf}
\DeclareOldFontCommand{\it}{\normalfont\itshape}{\mathit}
\DeclareOldFontCommand{\sl}{\normalfont\slshape}{\@nomath\sl}
\DeclareOldFontCommand{\sc}{\normalfont\scshape}{\@nomath\sc}

\def\Sym{{\rm Sym}}
\def\Z{{\mathbb{Z}}}
\def\N{{\mathbb{N}}}
\def\Q{{\mathbb{Q}}}
\def\R{{\mathbb{R}}}

\def\tp{{\rm tp}}
\def\Aut{{\rm Aut}}

\def\cl{{\rm cl}}

\def\LL{\mathcal{L}}

\def\hF{\widehat{F}}
\def\hR{\widehat{R}}
\def\Ann{{\rm Ann}}
\def\supp{{\rm supp}}
\def\Aug{{\rm Aug}^0}

\def\Q{\mathbb{Q}}
\def\N{\mathbb{N}}
\def\R{\mathbb{R}}
\def\Z{\mathbb{Z}}

\def\C{\mathcal{C}}

\def\h0{\overline{0}}

\def\l{\ell}

\usepackage[hidelinks]{hyperref}
\begin{document}
\hypertarget{first_page}{}\label{first_page}

\setheadsepline{1pt}[\color{black}]

\hskip 0.55 cm
\begin{tabular}{c r}
\vspace{-0.3cm}
\end{tabular}
\hskip 1.2 cm
\begin{tabular}{l}
\hline
\vspace{-0.1cm}
\footnotesize \rm\hspace{0.4cm} \bf Advances in Group Theory and Applications\\
\vspace{-0.1cm}
\scriptsize \hspace{0.47cm} $\copyright$ 20xx AGTA - \href{www.advgrouptheory.com/journal}{www.advgrouptheory.com/journal}\\
\vspace{-0.19cm}
\footnotesize\hspace{0cm} {\bf\rm xx} {\rm(}20yy{\rm)}, pp. \hyperlink{first_page}{\pageref*{first_page}}--\hyperlink{last_page}{\pageref*{last_page}} \scriptsize\rm\hfill\hspace{2.2cm} ISSN: 2499-1287\\
\vspace{0cm}
\scriptsize \rm\hfill\hspace{3cm} DOI: 10.4399/9788854890817xx \\
\hline
\end{tabular}
\hskip 1 cm
\begin{tabular}{c c}
\vspace{-0.1cm}
\vspace{-1cm}
\end{tabular}
\vspace{1.3 cm}

\titl{Permutation modules for Ramsey structures}

\auths{David M. Evans}
$$
  \begin{array}{c}
    \textnormal{\small Dedicated to Alan Camina on his 85th Birthday}
  \end{array}
$$
\thispagestyle{empty}
\justify\noindent
\setstretch{0.3}
\abstract{ Suppose $R$ is a commutative ring and $G$ is a group acting on a set $W$. We consider the $RG$-module $RW$ in the case where $G$ is the automorphism group of an $\omega$-categorical structure $M$ and $W$ is, for example, $M^n$ (for $n \in \mathbb{N}$). We develop methods which may provide information about two questions in the case where $R$ is a field $F$: whether $FW$ has a.c.c. on submodules; and in the case where $M$ is finitely homogeneous, whether $FW$ is of finite composition length. In the case where $M$ is a Ramsey structure and so $G$ is extremely amenable, we give a simple `decision procedure' for membership in a submodule of $RW$ specified by a given generating set. If $F$ is a field, we show that there is a duality between submodules of $FW$ and the topological $FG$-module of definable functions from $W$ to $F$.}

\setstretch{2.1}
\noindent
{\fontsize{10pt}{-10pt}\selectfont {\it Mathematics Subject Classification (2020)}: Primary 20B27; Secondary 20B07, 03C15.}\\[-0.8cm]

\noindent 
\fontsize{10pt}{-10pt}\selectfont  {\it Keywords}: Oligomorphic permutation groups; $\omega$-categorical structures; structural Ramsey theory; extreme amenability.\\[-0.8cm]

\setstretch{1.1}
\fontsize{11pt}{12pt}\selectfont

\section{Introduction}
\subsection{Background and Questions}
Suppose $G$ is a group acting  on a set $W$ and $R$ is a commutative ring with 1. Let $RW$ denote the free $R$-module with basis $W$. By extending the permutation action on the basis linearly, we obtain an action of $G$ on $RW$. Formally, this makes $RW$ into a module for the group ring $RG$, but we shall say simply that $RW$ is a $G$-module. We refer to this as a \textit{permutation module}.  We will represent an element $x$ of $RG$ as a formal sum
$x = \sum_{w \in W} \alpha_w w$
where the $\alpha_w$ are elements of $R$ and only finitely many of these are non-zero. The finite set $\{w : \alpha_w \neq 0\}$ is called the \textit{support} of $x$.

We are interested in submodules of this module in the case where $M$ is a countably infinite $\omega$-categorical structure, $G = \Aut(M)$ is the automorphism group of $M$ and $W$ is a sort in $M^{eq}$ (for example, $W$ is $M^n$ for some $n \in \N$: see Section \ref{MT} for more details). We will mainly be interested in the case where $R$ is a field $F$. In this case, $FW$ is the $F$-vector space with basis $W$ and $G$ is acting on this as $F$-linear maps. Our aim is develop tools in the case where $M$ is a \textit{Ramsey structure}. As will be clear from the proofs, there are strong similarities between what we are doing here and the use of canonical functions in work on infinite domain CSPs (see \cite{B} for a comprehensive account) and the decidability results in \cite{BPT}.

Permutation modules first arose in model theory in the work of Ahlbrandt and Ziegler \cite{AZ, AZ2} and in work on finite covers (see \cite{EIM} for a survey). Abstracting some ideas from \cite{AZ} and \cite{HrTC}, the paper \cite{CE} gave a number of general results and questions. Independently, permutation modules arising from $\omega$-categorical structures have also been studied more recently in the computer science literature: for example, the papers \cite{BFKM, GHL22}. The following question appears in \cite{CE}  and more recently in \cite{BFKM}:

\begin{question}\label{Q1} \rm Suppose $M$ is a countable $\omega$-categorical structure and $G = \Aut(M)$. If $W$ is a sort in $M^{eq}$, does $FW$ have the ascending chain condition (acc) on submodules (equivalently, is every submodule finitely generated)?
\end{question}

It would be enough to prove acc in the cases where $W$ is $M^n$.  A model-theoretic method for proving acc in Question \ref{Q1} is to show that $W$ has an \textit{AZ-enumeration} (the idea is  due to Ahlbrandt and Ziegler and is sometimes called a nice enumeration, or a good enumeration). Further details and references can be found in \cite{CE} or \cite{EIM}. The paper \cite{Ivanov} show that this method has limitations by giving  an example of an $\omega$-categorical structure \textit{without} an AZ-enumeration (answering a question in \cite{AC,CE}). 

It is known  that the descending chain condition for $FW$ can fail for general $W$ (in fact, for the case where $F$ is a finite field, $M$ is the vector space of countable dimension over $F$ and $W$ is the corresponding projective space: see \cite{AZ2, BFKM}). However, for structures $M$ which are homogeneous in a finite relational language we have the following stronger question from \cite{BFKM}:

\begin{question} \label{Q2} \rm Suppose that $M$ is a structure which is homogeneous in a finite relational language and $G = \Aut(M)$. Let $k \in \N$ and consider $G$ acting on $W = M^k$. Is $FW$ of finite length for every field $F$? 

\end{question}

Here, we say that a module $U$ has \textit{finite length} if there is a bound on the length $n$ of a chain of submodules $\{0\} < U_1 < U_2 \ldots < U_n = U$. The maximum length is then called the length of $U$. 

\begin{remarks} \rm It is not hard to show that a negative answer to Question \ref{Q2} with $F$ finite would yield:
\begin{itemize}
\item[(i)] a structure which is homogeneous for a finite relational language whose automorphism group has infinitely many closed normal subgroups; 
\item[(ii)] a structure which is homogeneous in a finite relational language which has infinitely many first-order reducts.
\end{itemize}
So (i) would give a negative answer to a question of Macpherson (Question 2.2.7 (4) in \cite{HDM}) and (ii) would give a counterexample to a well-known conjecture of Simon Thomas \cite{ST} on reducts of finitely homogeneous structures. 
\end{remarks}
For the current `state-of-the art' for Question \ref{Q2}, see \cite{YBK}.

\subsection{Ramsey structures and a decision procedure} The aim of this paper is to initiate a consideration of Questions \ref{Q1} and \ref{Q2} in the case where $M$ is a (reduct of a) \textit{$\omega$-categorical Ramsey structure}. Although there are $\omega$-categorical structures which do not have this property (see \cite{EHN}), those that do have it form a broad and robust class of interesting examples (see \cite{HN}, for example) which is conjectured to include all finitely homogeneous structures. This approach is clearly motivated by work on reducts and infinite domain CSP's (see \cite{B}, for example).

If $M$ is a (ordered) \textit{Ramsey structure}, then by the theorem of Kechris, Pestov and Todor\v cevi\'c (\cite{KPT}), the group $G= \Aut(M)$, considered as a topological group with the usual topology on a permutation group, is \textit{extremely amenable}: meaning that every non-empty compact space on which it acts continuously has a $G$-fixed point (see \ref{Ramseyreview}). Under this hypothesis, we can develop a duality between the module $FW$ and the topological module $F[W]$ of \textit{definable functions} $W \to F$. These are functions which are constant on the parts of a partition of $W$ into finitely many definable sets.  
This approach has several consequences which seem to us to be surprising, including the following `decision procedure'.

Suppose that $R$ is a commutative ring, $x, v_1,\ldots, v_r \in RW$ and we wish to decide whether or not $x$ is in $Y = \langle v_1,\ldots, v_r \rangle_{RG}$, the $RG$-submodule of $RW$ generated by $v_1,\ldots, v_r$. Let $S \subseteq W$ be the support of $x$ and let $G_{(S)}$ be the (pointwise) stabiliser of $S$ in $G$.  Let $W_1,\ldots, W_\ell$ be the $G_{(S)}$-orbits on $W$ (there are only finitely many such orbits, by the $\omega$-categoricity). Define a map 
\[ \Omega_S: RW \to R^\ell\]
by 
\[\Omega_S\left(\sum_{w\in W} \gamma_w w\right) = \left(\sum_{w\in W_1} \gamma_w ,\ldots, \sum_{w\in W_\ell} \gamma_w \right).\]
If $x \in Y$, then clearly $\Omega_S(x) \in \Omega_S(Y)$. In general, there is no reason to expect that the converse should hold. However, we show in Theorem \ref{mainresult} that if $M$ is an $\omega$-categorical Ramsey structure, then the converse does indeed hold. Using this we can reduce our original decision problem to a computation in the finite-rank, free $R$-module $R^\ell$. Let $v_1,\ldots, v_t$ be representatives for the $G_{(S)}$-orbits on the union of the $G$-orbits on $RW$ containing $v_1,\ldots, v_r$ (the $\omega$-categoricity of $M$ and finiteness of $S$ guarantee that there are finitely many such orbits). Then we have the following (of course, $t$ and $\ell$ here will increase with the size of the support of $x$). 

\begin{theorem} \label{mainresult} Suppose $R$ is a commutative ring, $M$ is an $\omega$-categorical Ramsey structure and $W$ is a sort in $M^{eq}$. Let $x, v_1,\ldots, v_t$, $S$, $\Omega_S$ be as defined above. Then

\[ x \in \langle v_1,\ldots, v_r\rangle_{RG} \Leftrightarrow \Omega_S(x) \in \langle \Omega_S(v_1),\ldots, \Omega_S(v_t)\rangle_R. \Box\]
\end{theorem}

This gives the following  `separation property':

\begin{corollary}  \label{sep} Let $F$ be a field, $M$ an $\omega$-categorical Ramsey structure and $W$ a sort in $M^{eq}$. Suppose $V \leq FW$ is an $FG$-submodule and $x \in FW \setminus V$. Let $S$ be the support of $x$. Then there is a $G_{(S)}$-invariant function $f : W \to F$ with $f(V) = 0$ and $f(x) \neq 0$.
\end{corollary}

\begin{remarks} \rm In the case $M = (\mathbb{Q}; \leq)$, these results have been proved using different methods (relying on the results of \cite{BFKM}) in \cite{AGThesis}. Decision processes for reducts of $\omega$-categorical Ramsey structures are proved in \cite{BPT}. The methods used here are clearly based on those in \cite{BPT}, but it is less clear whether our results formally follow from those in \cite{BPT}.
\end{remarks}

The main results Theorem \ref{mainresult} and Corollary  \ref{cfgen} are  proved in Section \ref{dd}. In Section~\ref{defdualsec} we show that, for Ramsey structures, we have a duality between submodules of $FW$ and closed submodules of the \textit{definable} dual $F[W] \leq F^W$ (Corollary \ref{defdual}). We then give some model-theoretic results about computing the topological closure in $F[W]$. In Section \ref{acc}, we  show (Corollary \ref{acccor})  that for a reduct  of an $\omega$-categorical Ramsey structure, a failure of acc in Question \ref{Q1} can be witnessed by a proper ascending chain of cyclic modules. Whilst this is perhaps surprising, it seems to be a long way from a proof of acc for such a structure.

The remainder of the current section is a summary of some of the notation and background we are using. It should perhaps more properly be regarded as an appendix which can be referred to if necessary.

\subsection{Background definitions and notation} \label{Appendix} In this section we collect together some facts, terminology and notation about permutation groups and model theory. The reference \cite{HDM} gives a much more comprehensive account of the background.

\subsubsection{Permutation groups} Suppose $G$ is a group acting (on the left) on a set $X$ and $A \subseteq X$. We denote by $G_{(A)}$ the pointwise stabiliser $\{g \in G : g(a) = a \mbox{ for all } a\in A\}$. If the action is faithful, then we call $G$ a permutation group (on $X$) and we may regard $G$ as a subgroup of $\Sym(X)$, the symmetric group of all permutations of $X$. In this case it is useful to think of $\Sym(X)$ as a topological group and give $G$ the relative topology where basic open sets are left cosets of subgroups $G_{(A)}$ for finite $A \subseteq X$.   Note that if $g \in G$, then $gG_{(A)} = \{ h \in G : h \vert A = g\vert A\}$, so the topology here is that of $G \subseteq X^X$ with the product topology where $X$ has the discrete topology (here and elsewhere, $g\vert A$ denotes the restriction of the function $g$ to $A$). The topology on $\Sym(X)$  is sometimes called the topology of pointwise convergence, and for countable $X$, it  is complete metrizable. We occasionally use this topology when the action is not faithful (in which case the topology is not Hausdorff and the closure of the identity subgroup is the kernel of the action).

With this topology, a subgroup of $\Sym(X)$ is closed iff it is the automorphism group of a first-order structure with domain $X$. In this case the structure can be taken to be relational with $G$-orbits on $X^n$ as the atomic $n$-ary relations: we refer to this as the canonical structure (or language) associated with $(G; X)$. A subgroup of $\Sym(X)$ is compact iff it is closed and all of its orbits on $X$ are finite (in which case, it is profinite). 

A subgroup $H$ of $G \leq \Sym(X)$ is open iff there is a finite set $A \subseteq X$ with $G_{(A)} \leq H$. Let $a$ be a tuple enumerating the elements of $A$ and let $D$ be the $G$-orbit containing $a$. For $g,h \in G$ write $E(ga,ha)$ iff $g^{-1}h \in H$.  This is a well defined, $G$-invariant equivalence relation on $D$ and the stabilizer of the $E$-class containing $a$ is $H$. 

Note that if $R$ is any commutative ring and $G \leq \Sym(X)$, then the action of $G$ on the permutation module $RX$, considered as a discrete space, is continuous.

\subsubsection{Model theory} \label{MT} The notation (and abuse thereof) is fairly standard. Suppose $\LL$ is a first-order language and $M$ an $\LL$-structure. We usually do not distinguish notationally between the structure and its domain (underlying set): thus $M$ will also denote the latter. We denote by $\Aut(M)$ the automorphism group of $M$. If $M$ is a countably infinite $\omega$-categorical structure (we usually just say `$\omega$-categorical' here) then by the Ryll-Nardzewski Theorem $G = \Aut(M)$ is \textit{oligomorphic} on $M$: it has finitely many orbits on $M^n$ for all $n \in \omega$ (and conversely). Equivalently, $G_{(A)}$ has finitely many orbits on $M$ for all finite $A \subseteq M$. In this case a subset of $M^n$ is $A$-definable iff it is $G_{(A)}$-invariant and $n$-tuples $b,c \in M^n$ are in the same $G_{(A)}$-orbit iff they have the same (complete) type over $A$. We denote the latter condition by the usual model-theoretic notation $\tp(b/A) = \tp(c/A)$, or $b \equiv_A c$. We also use this notation (with some care) when $a, b$ are elements of other sets on which $G$ is acting continuously, such as a permutation module $FW$ or the module of definable functions $F[W]$.

If $D$ is a $\emptyset$-definable subset of $M^n$ and $E$ is a $\emptyset$-definable equivalence relation on $D$, we refer to the equivalence classes as \textit{imaginary elements} of $M$ (of sort $D/E$). We denote the set of these (varying $D, E$, and including $M$) as $M^{eq}$. If $G = \Aut(M)$, then $G$ acts continuously on $M^{eq}$ and if $e \in M^{eq}$, then its stabilizer $G_e$ is open in $G$. If $M$ is (countable) $\omega$-categorical then the converse is true: any open subgroup of $G$ is of the form $G_e$ for some $e \in M^{eq}$. Up to interdefinability, $M^{eq}$ is closed under finite unions of sorts. So being a sort in $M^{eq}$ is the same thing as being a finite union of $G$-orbits on $M^{eq}$.

If $M$ is $\omega$-categorical, $a \in M$ and $B$ is a finite subset of $M$, the model-theoretic notions of $a$ being in the algebraic (resp. definable) closure of $B$ correspond to $a$ being in a finite $G_{(B)}$-orbit (resp. of size 1). This extends to $M^{eq}$ in an obvious way. 

\def\Age{{\rm Age}}
\def\Fraisse{Fra\"{\i}ss\'e }

Suppose that $\LL$ is a (relational)  language and $M$ is a countable $\LL$-structure. The class of finite $\LL$-structures which can be embedded into $M$ is referred to as the \textit{age} of $M$, denoted by $\Age(M)$. We say that $M$ is a \Fraisse structure if every isomorphism between finite substructures of $M$ extends to an automorphism of $M$ (other terminology used is that $M$ is \textit{homogeneous}, or \textit{ultrahomogeneous}). We can make the same definition for arbitrary structures $M$ (where function symbols are in the language), but for the purposes here, we should assume local finiteness of $M$: every finite subset is contained in a finite substructure. This holds automatically if $M$ is $\omega$-categorical. The classical result of \Fraisse says that $M$ is a \Fraisse structure iff $\Age(M)$ is an amalgamation class. Moreover, every amalgamation class of finite structures with countably many isomorphism types is the age of a (unique up to isomorphism) countable \Fraisse structure, referred to as the \Fraisse limit of the class. 

Note that if $G$ a closed permutation group on a countable set $X$, then the canonical structure associated with $(G;X)$ is a \Fraisse structure with automorphism group $G$. So any general result about automorphism groups of \Fraisse structures is  a general result about closed permutation groups on a countable set.

In Section \ref{defdualsec} we will use the notion of an \textit{elementary embedding} between two $\LL$-structures. We recall the definition here and refer the reader to a standard text in model theory for more details (for example, \cite{Mar}). Suppose $M, N$ are $\LL$-structures and $M$ is a substructure of $N$. We say that $M$ is an \textit{elementary substructure} of $N$ (and $N$ is an \textit{elementary extension} of $M$) if whenever $\phi(x_1,\ldots, x_n)$ is an $\LL$-formula and $a_1,\ldots, a_n \in M$, then $M \models \phi(a_1,\ldots, a_n)$ iff $N \models \phi(a_1,\ldots, a_n)$. In other words, the tuple $(a_1,\ldots, a_n)$ satisfies the formula $\phi$ in $M$ iff it satisfies $\phi$ in $N$. In this case we write $M \preceq N$. More generally an embedding $\alpha : M \to M'$ between $\LL$-structures is an \textit{elementary embedding} if its image is an elementary substructure of $M'$. 

Note that in the above we can take $\phi$ to be a closed formula and so $M \preceq N$ implies that $M, N$ are elementarily equivalent. In general the converse does not hold. However, if $M$ is a countable, $\omega$-categorical homogeneous structure, then its theory has quantifier elimination and embeddings between models are elementary.

\subsubsection{Ramsey classes and extreme amenability}  \label{Ramseyreview} A topological group $G$ is said to be \textit{extremely amenable} if, whenever $Y$ is a non-empty, compact Hausdorff space on which $G$ acts continuously, then there is a a $G$-fixed point in $Y$ (in this action). In the case where $G$ is a closed permutation group on a countably infinite set, we have the following result of Kechris, Pestov and Todor\v cevi\'c (\cite{KPT}):

\begin{theorem} Suppose that $M$ is a countable \Fraisse structure in a language $\LL$ and $G = \Aut(M)$. Then $G$ is extremely amenable if and only if $\Age(M)$ is a Ramsey class of rigid finite $\LL$-structures.
\end{theorem}

Here, a finite $\LL$-structure $A$ is said to be rigid if it has trivial automorphism group. A class $\C$ of finite $\LL$-structures is said to be a \textit{Ramsey class} if it is closed under isomorphisms, taking substructure and satisfies the following property (for all $k \in \N$, but it is sufficient to verify this for $k=2$):

For all  $A \subseteq B \in \C$ there is $B \subseteq C \in \C$ such that whenever $f : \binom{C}{A} \to [k]$ is a $k$-colouring of the copies of $A$ inside $C$, then there is a copy $B'$ of $B$ contained in $C$ such that $f$ restricted to $\binom{B'}{A}$ is constant.

If $\C$ is a Ramsey class and has the joint embedding property, then $\C$ is an amalgamation class. So $\C$ is the age of a \Fraisse structure. We refer to a countable \Fraisse structure whose age is a rigid Ramsey class as a \textit{Ramsey structure}.

 Verifying that $M$ is a Ramsey structure usually involves proving that its age is a Ramsey class using combinatorial methods (cf. \cite{HN} for recent advances in the techniques for doing this). Here, we will be be working with Ramsey structures and using extreme amenability of the automorphism group as the key property, rather than the Ramsey property of the age. Examples of Ramsey classes can often be described by specifying an amalgamation class $\C_0$ of finite structures and then describing $\C$ as the class of structures in $\C_0$ expanded by certain orderings. For example, the following give rise to Ramsey classes:
 \begin{itemize}
 \item $\C_0$ is a free amalgamation class in a countable relational language $\LL_0$ and $\C$ is obtained by considering all orderings of structures in $\C_0$;
 \item $\C_0$ is the class of all finite partial orderings and $\C$ is obtained by taking elements of $\C_0$ expanded by a compatible linear ordering;
 \item $\C_0$ is the class of finite-dimensional vector spaces over a fixed finite field $E$. The orderings we add are \textit{locally reverse lexicographic} (see \cite{T2}, for example).
 \end{itemize}
Respectively, the Ramsey structures which are the \Fraisse limits of these classes are sometimes referred to as: the generically ordered \Fraisse limit of $\C_0$; the compatibly ordered generic poset; the generic locally reverse lexicographically ordered $E$-vector space. More examples and information can be found in the references  \cite{HDM, B, EHN, HN}.

We will make much use of the following result:

\begin{theorem}\label{openisea} (\cite{BPT}, Lemma 13) Suppose that $G$ is an extremely amenable topological group and $H$ is an open subgroup of $G$. Then $H$ is also extremely amenable. 
\end{theorem}

\textbf{Acknowledgements} Thanks are due to Miko{\l}aj Boja\'nczyk, Arka Ghosh, Piotr Hofman, Bartek Klin, S{\l}awomir Lasota, Szymon Toru\' nczyk and  Jingjie Yang for interesting discussions about their work and some of the material in this paper.

\section{Dualities}\label{dualitysection}

\subsection{Vector space duality}

\begin{definition} \rm Suppose $R$ is a commutative ring and $G$ is a group acting on a set $W$. Let $R^W$ denote the $R$-module of functions $f : W \to R$, regarded as a topological space with the pointwise convergence topology (where $R$ has the discrete topology; so for any finite $S \subseteq W$ and $X\subseteq R^S$, the set $X \times R^{W\setminus S} \subseteq R^W$ is clopen). This is a $G$-module under the action $(gf)(w) = f(g^{-1}(w))$. We refer to $R^W$ as the \textit{dual permutation module} (associated with $(G; W)$). 

Note that if $f \in R^W$ and $v = \sum_w \alpha_w w \in RW$, then we may define $f(v) = \sum_w \alpha_w f(w)$ as only finitely many $\alpha_w \in R$ are non-zero. If $g \in G$, then we have $(gf)(v) = f(g^{-1}v)$.
\end{definition}

We now assume that $R$ is a field $F$. In this case, the terminology is explained by the following well-known fact. The notation is as above.
\begin{lemma} \label{duality} (i) If $U$ is a subspace of $FW$, then $U^\perp = \{ f \in F^W :  f(u) = 0 \mbox{ for all } u \in U\}$ is a closed subspace of $F^W$. If $U$ is $G$-invariant, then so is $U^\perp$.

(ii) If $V$ is a subspace of $F^W$, then $V^\perp = \{ x \in FW : f(x) = 0 \mbox{ for all } f \in V\}$ is a subspace of $FW$. Moreover, $(V^\perp)^\perp$ is the closure of $V$ in $F^W$. If $V$ is $G$-invariant, then so is $V^\perp$.

(iii) There is an inclusion-reversing correspondence between the poset of $G$-submodules of $FW$ and the poset of closed $G$-submodules of $F^W$.
\end{lemma}

\textit{Proof:\/} (iii) follows from (i) and (ii) and the statements about $G$-invariance are clear.

(i) As $U^\perp = \bigcap_{v \in U} \langle v \rangle^\perp$ it suffices to say why $\langle v \rangle^\perp$ is closed. The support $S$ of $v$ is finite and whether $f(v) = 0$ depends only on $f\vert S$. So the set of such $f$ is closed in $F^W$.

(ii) Clearly $V \subseteq (V^\perp)^\perp$ and by (i), the latter is closed. Let $V_1$ denote the closure of $V$ and suppose, for a contradiction, that $V_1 < (V^\perp)^\perp$. As $V_1$ is closed, there is a finite $S \subseteq W$ such that $V(S) \leq V_1(S) < (V^\perp)^\perp(S)$, where the $S$ here denotes restrictions to $S$. So this is a statement about the finite dimensional space $F^S$. As $V_1(S) = V(S)$ (by definition of closure), this contradicts the well-known duality of finite dimensional spaces (in particular, between $FS$ and $F^S$). $\Box$

Note that (ii) in the above uses that $F$ is a field (rather than an arbitrary commutative ring).

\subsection{Proof of the Main Results}\label{dd}

Throughout, $M$ is a countable $\omega$-categorical structure, $G = \Aut(M)$ and $W$ is a sort in $M^{eq}$. In particular, $G$ acts oligomorphically and continuously on $W$. One of the main aims of this subsection is to prove:

\begin{corollary} \label{cfgen} Suppose that $M$ is Ramsey, and $F$ is a field. Suppose that $U < V \leq F^W$ are closed, $G$-invariant subspaces of $F^W$. Then there is a finite $S \subseteq W$ and a $G_{(S)}$-invariant function $h \in V \setminus U$.
\end{corollary}

Note that Corollary \ref{sep} follows from this and the duality in the previous section. We first indicate the proof of \ref{cfgen} in the case where $F$ is a finite field. The proof of the general case is similar, but requires more machinery. So suppose that $F$ is finite and therefore  $F^W$ is compact. As $U$ is closed in $V$, there is some finite $S \subseteq W$ such that $f\vert S$ is not the restriction of any function in $U$ to $S$.  It suffices to show that there is some $G_{(S)}$-invariant  $h \in  V$ with $h \vert S = f\vert S$ (as this gives $h \not\in U$). Consider $Y = \{ h \in V : h \vert S = f\vert S\}$. This is a non-empty compact space on which the group $G_{(S)}$ is acting continuously. As $G_{(S)}$ is an open subgroup of the extremely amenable group $G$, it is also extremely amenable (Theorem \ref{openisea}). So there is a $G_{(S)}$-fixed element $h$ of $Y$. $\Box$

In order to prove the general version of Corollary \ref{cfgen}, we will make use of  Pontryagin duality between discrete and compact abelian groups (\cite{P}, Chapter 6). 

Suppose that $A = (A, +)$ is an abelian topological group. A \textit{character}  of $A$ is continuous homomorphism $\chi: (A,+) \to (\R/\Z , +)$, where $\R/\Z$ is the topological group $S^1$, the circle group written additively. Denote by $\h0$ the zero element of $\R/\Z$. Under addition and with the compact open topology for functions $A \to \R/\Z$, the characters of $A$ form a topological abelian group $\hat{A}$, called the \textit{Pontryagin dual} of $A$. If $f: A \to B$ is a continuous homomorphism of topological abelian groups, then we obtain a continuous homomorphism $\widehat{f} : \widehat{B} \to \widehat{A}$ (where $\widehat{f}(\chi) = \chi\circ f$). 

If $A$ is discrete, then $\widehat{A}$ is compact Hausdorff; conversely if $A$ is compact (and Hausdorff), then $\widehat{A}$ is discrete. Pontryagin duality tells us that in these cases, the double dual $\widehat{\widehat{A}}$ is naturally isomorphic to $A$ (if $a\in A$ then the evaluation map $\chi \mapsto \chi(a)$ is a character of $\widehat{A}$; the result is that this map gives an isomorphism $A \to \widehat{\widehat{A}}$). In fact, dualising in this way gives us a contravariant isomorphism between the categories of discrete abelian groups and of compact Hausdorff abelian groups (and continuous homomorphisms).

Now suppose $R$ is  a commutative ring. We consider $(R, +)$ as a discrete topological group, so the Pontryagin dual $(\hR, +)$ of  $(R,+)$ is a compact topological group. 
As $R$ is a commutative ring, we may consider $\hR$ as a left $R$-module by setting $(r\chi)(q) = \chi(rq)$, for $r,q \in R$ and $\chi \in \hR$.

Let $W$ and $G$ be as previously and let  $RW$ denote the free $R$-module with basis $W$. We consider this with the discrete topology; its dual (as an abelian group) is $\hR^W$ with the product topology. This is compact and the $G$-action on $\hR^W$ is continuous.

We have the pairing $(.,.) : \hR^W \times RW \to \R/\Z$ given by
\[ (f, \sum_{w \in W} \alpha_w w) = \sum_{w\in W } f(w)(\alpha_w).\]

Note that this is $G$-invariant: if $g \in G$, then 
\[ (gf, g\left(\sum_{w \in W} \alpha_w w\right)) = (f, \sum_{w \in W} \alpha_w w).\]

The duality tells us that if $X$ is a subgroup of $RW$ then its annihilator $\Ann(X) = \{ f \in \hR^W : (f,x) = \h0 \mbox{ for all } x \in X\}$ is a closed subgroup of $\hR^W$. Moreover, the map $X \mapsto \Ann(X)$ gives an inclusion-reversing correspondence between subgroups of $RW$ and closed subgroups of $\hR^W$ which restricts to a correspondence between $RG$-submodules of $RW$ and closed $RG$-submodules of $\hR^G$.

We will also use the following terminology:

\begin{definition}\label{support} \rm Suppose $u = \sum_{w} \alpha_w w  \in RW$. The \textit{support} of $u$ is the finite set $\supp(u)  = \{ w \in W : \alpha_w \neq 0\}$. If $C$ is a finite subset of $M^{eq}$, we say that the support of $u$ is \textit{contained in $C$} to mean that $u$ is fixed by $G_{(C)}$.
\end{definition}

\begin{proposition}  \label{hcf} Suppose that $M$ is a  Ramsey structure and therefore $G$ is extremely amenable. Let $X$ be a $G$-invariant subgroup of $RW$ and let $u \in RW$ have support contained in $S \subseteq_{fin} W$. Then $u \in X$ if and only if whenever $k \in \hR^W$ is $G_{(S)}$-invariant and $(k,x) = \h0$ for all $x \in X$, then we have $(k,u) = \h0$.
\end{proposition}

\textit{Proof:\/} The direction $\Rightarrow$ is trivial. So suppose $u \not\in X$. By the duality (using that $\Ann(\langle X, u\rangle) < \Ann(X)$) there exists $f \in \hR^W$ such that $(f,X) = \h0$ and $\delta = (f,u) \neq \h0.$ Consider 
\[ H = \{ h \in \hR^W : (h,X) = \h0 \mbox{ and } (h, u) = \delta\}.\]
This is closed in $\hR^W$ and therefore compact;  it is non-empty and also $G_{(S)}$-invariant. As $G_{(S)}$ is extremely amenable there is therefore a $G_{(S)}$-invariant $k \in H$, as required. $\Box$

\begin{definition}\rm Suppose $S \subseteq M^{eq}$ is finite. List the  $G_{(S)}$-orbits on $W$ as $O_1,\ldots, O_\ell$. We define a map $\Omega_S : RW \to R^\ell$ in the following way. For $z = \sum_w \alpha_w w \in RW$, let $\Omega_S(z) = (\gamma_1,\ldots, \gamma_\ell)$ where $\gamma_i = \sum_{w \in O_i} \alpha_w$.
\end{definition}

\begin{remarks}\rm We can view this as a generalised augmentation map. Indeed, as an $RG_{(S)}$-module we have
$RW = \bigoplus_{i = 1}^\ell RO_i$
and $\Omega_S$ is the direct sum of the augmentation maps $\eta_i: RO_i \to R$ (where $\eta_i(\sum_{w \in O_i} \alpha_w w) = \sum_{w \in O_i} \alpha_w$).
\end{remarks}

\begin{theorem}\label{omegathm} Suppose that $M$ is an $\omega$-categorical  Ramsey structure. Let $X$ be a $G$-invariant subgroup of $RW$ and suppose that $u \in RW$ has support contained in $S \subseteq_{fin}  M^{eq}$.  Then 
\[ u \in X \Leftrightarrow \Omega_S(u) \in \Omega_S(X).\]
\end{theorem}

\textit{Proof:\/} Again, $\Rightarrow$ is trivial. Suppose $h \in \hR^W$ is $G_{(S)}$-invariant. List the $G_{(S)}$-orbits on $W$ as $O_1,\ldots, O_\ell$. We can write 
$h = \sum_{i = 1}^\ell \phi_i \mathbf{1}_{O_i}$, where $\phi_i \in \hR$ and $\mathbf{1}_{O_i}$ is the characteristic function of $O_i$. Let $[h]_S = (\phi_1, \ldots, \phi_\ell) \in \hR^\ell$. 

For $(x_1,\ldots, x_\ell) \in R^\ell$ define 
\[(\phi_1,\ldots, \phi_\ell)*(x_1, \ldots, x_\ell) = \phi_1(x_1) + \ldots + \phi_l(x_\ell) \in \R/\Z.\]
Then for any $z \in RW$ we have
\[ (h, z) = [h]_S * \Omega_S(z).\]
Thus, if $\Omega_S(u) \in \Omega_S(X)$ and $(h,X) = \h0$, then $(h,u) = \h0$. It then follows by Proposition \ref{hcf} that $u \in X$. $\Box$

We can now give:

\textit{Proof of Corollary \ref{cfgen}:\/} By Lemma \ref{duality} there are $G$-invariant subspaces $X < Y \leq FW$ with $V = X^\perp$ and $U = Y^\perp$. Let $y \in Y \setminus X$ and let $S$ contain the support of $y$. As $y \not\in X$, Theorem \ref{omegathm} implies that $\Omega_S(y) \not\in \Omega_S(X)$.

Denote by $\cdot$ the usual dot-product on $F^\ell$. This is a non-degenerate symmetric bilinear form on $F^\ell$ and therefore (by standard finite-dimensional linear algebra) there is $v \in F^\ell$ such that $v\cdot \Omega_S(x) = 0$ for all $x \in X$ and $v\cdot \Omega_S(y) \neq 0$. Let $v = (\alpha_1,\ldots, \alpha_\ell)$, denote the $G_{(S)}$-orbits on $W$ by $O_1,\ldots, O_\ell$ and let $h = \sum_{i=1}^\ell \alpha_i\mathbf{1}_{O_i} \in F^W$. This is $G_{(S)}$-invariant and  for any $z \in FW$ a simple calculation shows that $h(z) = v\cdot \Omega_S(z)$. It follows that $h \in X^\perp = V$ and $h(y) \neq 0$, so $h \not\in Y^\perp = U$, as required. $\Box$

We now return to the decision problem mentioned in the Introduction and give a proof of Theorem \ref{mainresult}.

Suppose that $M$ is an $\omega$-categorical Ramsey structure and $R$ is a commutative ring. Then Theorem \ref{omegathm} reduces questions about membership in an $RG$-submodule of $RW$ to a question about membership in some $R$-submodule of $R^\ell$. Indeed, suppose $x,v_1,\ldots, v_r \in RW$ are given and $S$ is a finite subset of $M^{eq}$ which contains the support of $x$. Let  $Y = \langle v_1,\ldots, v_r \rangle_{RG}$ be the $RG$-submodule of $RW$ generated by $v_1,\ldots, v_r$. Let $v_1,\ldots, v_t$ be representatives for the $G_{(S)}$-orbits on the union of the $G$-orbits containing $v_1,\ldots, v_r$. Thus $Y = \langle v_1,\ldots, v_t \rangle_{RG_{(S)}}$, the $RG_{(S)}$-submodule of $RW$ generated by $v_1, \dots, v_t$. 

Now, the map $\Omega_S : RW \to R^\ell$ is an $RG_{(S)}$-homomorphism (where $G_{(S)}$ acts trivially on $R^\ell$), therefore $\Omega_S(Y) = \langle \Omega_S(v_1),\ldots, \Omega_S(v_t)\rangle_R$. It then follows from Theorem \ref{omegathm} that $x \in Y$ iff $\Omega_S(x) \in \langle \Omega_S(v_1),\ldots, \Omega_S(v_t)\rangle_R$. So Theorem \ref{mainresult} is proved.

\begin{remarks}\rm Using Theorem \ref{mainresult}, we can reduce other computational problems about  $RW$ to `linear algebra' questions in $R^\ell$ and computations about $M$ essentially involving its universal theory (to compute orbits over finite sets). For example:
\begin{itemize}
\item Given $v_1,\ldots, v_r \in RW$, is $\langle v_1,\ldots, v_r\rangle_{RG} = RW$? \newline
To do this, take $w_1, \ldots, w_s$ representatives for the $G$-orbits on $W$ and test whether the vectors $w_1,\ldots, w_s \in RW$ are all in $\langle v_1,\ldots, v_r\rangle_{RG}$.
\item Given $v_1,\ldots, v_r \in RW$ and $k \in \N$, does  $\langle v_1,\ldots, v_r\rangle_{RG}$ contain a non-zero vector of support size at most $k$? \newline
By the $\omega$-categoricity, there is  a finite subset $S$ of $W$ such that every $G$-orbit on $k$-subsets of $W$ has a representative contained in $S$. By Theorem \ref{mainresult} the question is then equivalent to asking whether $\langle \Omega_S(v_1), \ldots, \Omega_S(v_r)\rangle_R \leq R^\ell$ has a non-zero vector $u$ where: there are at most $k$ non-zero coordinates in $u$; and these coordinates are ones corresponding to elements of $S$.
\end{itemize}

\end{remarks}

\begin{remarks}\label{reduct} \rm Suppose $M$ is an $\omega$-categorical Ramsey structure and $N$ is a reduct of $M$. Let $G = \Aut(M)$ and $H = \Aut(N)$. So $G$ is a closed, extremely amenable subgroup of $H$. Suppose  $W$ is a sort in $N^{eq}$ ($\subseteq M^{eq}$) and consider the $FH$-permutation module $FW$. Let $u, v_1,\ldots, v_s \in FW$ and consider the problem of determining whether or not $u$ is in $\langle v_1,\ldots, v_s\rangle_{FH}$. Consider the $G$-orbits on the union of the $H$-orbits (on $FW$) containing the vectors $v_1,\ldots, v_s$. There are finitely many of these (as $G$ acts oligomorphically on $W$); let $v_1,\ldots, v_t$ be representations from these $G$-orbits. It is easy to show that 
\[\langle v_1,\ldots, v_s\rangle_{FH} = \langle v_1,\ldots, v_t\rangle_{FG}.\]
As we have a method for deciding membership in the second of these, we have a method for the given $FH$-module.
\end{remarks}

\begin{remarks} \rm Similar results about decidability of membership are proved by different methods in \cite{GHL22} and \cite{AGThesis}. Theorem 6.1 of \cite{GHL22} proves the result for the structure of pure equality $(M; =)$ and Theorem 4.20 of \cite{AGThesis} proves this result in the case where $M$ is $(\Q; <)$. Theorem 4.33 of \cite{AGThesis} is precisely our Theorem \ref{omegathm} (for the case of $(\Q; <)$) and in \cite{AGThesis} this is proved  using the structure of the permutation modules $F\Q^{(n)}$ given in \cite{BFKM}. 
\end{remarks}

\begin{remarks} \rm Continue with the notation used in Remarks \ref{reduct}. We would like to know whether the separation property of Corollary \ref{sep} holds for the reduct $N$. Specifically, suppose $V \leq FW$ is $H$-invariant and $x \in FW \setminus V$. Is there a finite $S \subseteq N$ and an $H_{(S)}$-invariant function $f : W \to F$ such that $f(V) = 0$ and $f(x) \neq 0$. We know that there is a $G_{(S)}$-invariant such function (where $S$ is a support of $x$ in $M$), by Corollary \ref{sep}, but here we want $f$ to be definable in $N$. If $F$ has non-zero characteristic, we suspect that the answer is negative, so we would first suggest to look at the question in characteristic 0.
\end{remarks}

\section{Further results}

 \subsection{The definable dual}\label{defdualsec}
In this subsection  $F$ is a field, $M$ is a countable $\omega$-categorical structure, $G = \Aut(M)$ and $W$ is a sort in $M^{eq}$. So $G$ is acting continuously and oligomorphically on $W$. Some of these conditions can be weakened in what follows, but the terminology would then be less appropriate.

\begin{definition}\rm  

(i) Recall that  a subset $X \subseteq W$ is \textit{definable} (or $C$-definable) if there is a finite  $C \subseteq M$ such that $X$ is $G_{(C)}$-invariant (where the latter denotes the pointwise stabiliser of $C$ in $G$). 

(ii) A function $f \in F^W$ is \textit{definable} if it is a linear combination of characteristic functions of definable sets. In other words, there is a finite $C \subseteq M$ such that $f$ is constant on $G_{(C)}$-orbits.

(iii) Denote by $F[W]$ the (topological) subspace of $F^W$ consisting of definable functions. This is clearly $G$-invariant.
\end{definition}

Without the topology, the vector space of definable functions is also studied in \cite{GHL22} (where it is denoted by {\sc Lin}$_F(W)$). It is important to stress that we  think of $F[W]$ as an $F$-vector space with a topology and a continuous  $G$-action on it. In particular, the notion of closure on $F[W]$ includes topological closure. More precisely:

\begin{definition} \rm Suppose $Y \subseteq F^W$. We write $\cl^+(Y)$ for the smallest $G$-invariant closed subspace of $F^W$ containing $Y$. So this is the closure of $\langle Y\rangle_{FG}$ the $FG$-submodule generated by $Y$, the linear span $\langle G.Y\rangle_F$ of the union of $G$-orbits containing elements of $Y$. Note that for $f \in F^W$ we have $f \in \cl^+(Y)$ iff for every finite $S \subseteq W$ we have that $f\vert S$ is the restriction to $S$ of an element of $\langle Y\rangle_{FG}$.

If $Y \subseteq F[W]$ we write $\cl(Y) = \cl^+(Y) \cap F[W]$.
\end{definition}

\begin{remarks} \rm As $F[W]$ contains characteristic functions of points of $W$, it is dense in $F^W$. Moreover, via an identification of elements of $W$ with their characteristic functions, we can, if we wish, regard $FW$ as a subspace of $F[W]$. We will not make use of this.
\end{remarks}

Using Corollary \ref{cfgen} we have the following additional duality:

\begin{corollary} \label{defdual} Suppose that $M$ is Ramsey. Then the maps 
\[ U \mapsto U^\perp \leq  FW \]
and
\[X \mapsto X^\perp \cap F[W]\]
give an inclusion-reversing correspondence between the closed $G$-invar\-iant subspaces $U$ of $F[W]$ and $G$-submodules $X$ of $FW$.
\end{corollary}

\textit{Proof:\/}  Suppose $U < V$ are closed subspaces of $F[W]$ and let $U_1, V_1$ denote their closures in $F^W$. Then $U_1^\perp = U^\perp$ (by Lemma \ref{duality}) and similarly $V_1^\perp = V^\perp$. Moreover $U_1 < V_1$, so $V_1^\perp < U_1^\perp$, as required.

Conversely suppose $X < Y$ are $G$-invariant subspaces of $FW$. Then $Y^\perp < X^\perp$ (by \ref{duality}) and Corollary \ref{cfgen} shows that $Y^\perp \cap F[W] < X^\perp \cap F[W]$, as required. $\Box$

\bigskip

We will now use some model theory to provide more information on the closure $\cl^+(f)$ for $f \in F^W$.  We have been considering $W$ as a sort in $M^{eq}$. Formally, the sort is associated to the theory of $M$ and we now want to vary the model $M$. When we do this, we will write $W = W(M)$, the elements of the sort in $M^{eq}$. If $\alpha : M \to M'$ is an elementary embedding (see Section \ref{MT}), then we obtain a map between the corresponding sorts $W(M)$ and $W(M')$ which is also elementary and which we will also denote by $\alpha$. 

In the case where $\alpha$ is an elementary embedding, we can regard $W(M) \to W(M')$ as an inclusion map, and we have a corresponding map $F^{W(M')} \to F^{W(M)}$, given by restriction. If $f \in F^{W(M)}$ is $C$-definable (for some finite $C \subseteq M$), then there is a finite partition of $W(M)$ into $C$-definable sets on which $f$ takes constant values. There is a corresponding partition of $W' = W(M')$ into $C$-definable sets and a corresponding $C$-definable function $f' \in F^{W(M')}$ with $f'\vert W = f$. As a definable object, we can, with a little care,  regard $f$ itself as an element of $M^{eq}$ and talk, for example, about its type. Note  that with this convention, the image of $f$ under the elementary embedding $\alpha$ is actually $f'$. We now consider what happens when we restrict general elements of $F[W']$ (not necessarily defined over subsets of $M$) to $W$.

\begin{definition} \label{exdef}\rm Suppose $M \preceq M'$ is a (countable)  elementary extension and let $W' = W(M')$. Suppose that $f' \in F^{W'}$ is definable  and let  $h = f' \vert W$. Then we say that $h \in F^W$ is \textit{externally definable}. 

\end{definition}

\begin{lemma}  \label{elemb} Suppose $M$ is countable, $\omega$-categorical and $W = W(M)$ is a sort in $M^{eq}$.

(i) Suppose $f \in F[W]$ and $M'$ is an elementary extension of $M$. Let $W' = W(M')$. Then $f$ extends naturally to an element of $F[W']$ (which we still denote by $f$). Let $f' \in F[W']$ have the same type over $\emptyset$ as  $f$ in $M'$ and $h = f'\vert W$. Then $h \in \cl^+(f)$. 

(ii) Suppose  $\beta : W \to W$ is an elementary embedding. Let $f \in F^W$ and  let $h \in F^W$ be given by $h(w) = f(\beta(w))$. Then $h \in \cl^+(f)$. If $f$ is definable, then $h$ is externally definable.
\end{lemma}

\textit{Proof:\/} (i) We may suppose that $f, f'$ are defined over the finite sets $C \subseteq M$ and $C' \subseteq M'$, with $C, C'$ of the same type over $\emptyset$.  As $f$ is essentially defined by a formula with parameters in $C$, it  extends naturally to an element of $F[W']$, defined by the same formula. Let $X$ be a finite subset of $W$. As $M \preceq M'$ (and using the $\omega$-saturation of $M$ coming from $\omega$-categoricity) there is $C'' \subseteq M$ such that $C'', C'$ have the same type over $X$. Let $f''$ be the translate over $C''$ of $f'$ (i.e. the type of $f'C'$ is the same as the type of $f''C''$). Then $f'' \vert X = f' \vert X = h \vert X$. Moreover, as $f'', f$ are defined over parameter sets in $W$ with the same type, they have the same type in $M$ over $\emptyset$, so we have $f'' \in \langle f\rangle_{FG}$. It follows that $h \in \cl^+(f)$, as required.

(ii) Let $X$ be a finite subset of $W$. As $\beta$ is an elementary embedding and $M$ is strongly $\omega$-homogeneous, there is $\alpha \in \Aut(M)$ with $\beta\vert X = \alpha\vert X$. So for all $x \in X$ we have $h(x) = (\alpha^{-1}f)(x)$. Thus $h \in \cl^+(f)$.  For the external definability, change the notation slightly and regard $\beta$ as the inclusion map of an elementary extension $W \preceq W'$, and $f \in F[W']$. Then $h$, as defined, is the restriction of $f$ to $W$, so it is externally definable. $\Box$

\subsection{A (possible) reduction for a.c.c.} \label{acc}

First, we note the following  observations about an extremely amenable group acting on compact topological modules. (All spaces are assumed to be Hausdorff.)

\begin{lemma} \label{exactness} Suppose $G$ is an extremely amenable topological group. If $A$ is a continuous $G$-module (meaning that $A$ is a topological abelian group and the action $G \times A \to A$ is continuous), let $A^G = \{ a\in A : ga = a \,\mbox{ for all } g \in G\}$. 
\begin{enumerate}
\item[(i)] Let $0 \to A \stackrel{\alpha}{\to} C \stackrel{\beta}{\to} B \to 0$ be an exact sequence of compact, continuous $G$-modules (and continuous homomorphisms). Then the induced sequence $0 \to A^G \to C^G \to B^G \to 0$ is also exact.
\item[(ii)] Suppose $A$ is a continuous $G$-module and $A_1, \ldots , A_n$ are compact submodules of $A$. Then $(A_1 + \ldots + A_n)^G = A_1^G + \ldots + A_n^G$.
\end{enumerate}

\end{lemma}

\textit{Proof:\/} (i) Exactness of $0 \to A^G \to C^G \to B^G$ is true in general. Indeed,  if $c \in C^G$ is in $\ker\beta$,  then by exactness of the original sequence there is a unique $a \in A$ with $\alpha(a) = c$. As  $c\in C^G$, we have $a \in A^G$. For the rest, we may assume $A \leq C$ and $B = C/A$. It remains to show that if $c + A \in B^G$, then we may take $c \in C^G$. But $c+A$ is a compact, non-empty, $G$-invariant subset of $C$, so this follows immediately from extreme amenability.

(ii) It suffices to prove this in the case $n = 2$. Clearly $A_1^G+A_2^G \subseteq (A_1 + A_2)^G$. Let $a \in (A_1 + A_2)^G$ and write $a = a_1 + a_2$ with $a_i \in A_i$. So $a_1 \in A_1 \cap (a+A_2)$. This set is therefore a non-empty, compact, $G$-invariant subset of $A_1$, so there is $a_1' \in A_1 \cap (a+A_2)$ which is fixed by $G$. Then $a_2' = a - a_1' \in A_2$ is fixed by $G$ and so $a \in A_1^G+A_2^G$. $\Box$

\bigskip

Now suppose throughout this section that $M$ is a countable $\omega$-categorical structure and $F$ is a field. Let $G = \Aut(M)$.
Suppose $W \subseteq M^{eq}$ is a finite union of $G$-orbits $W_1,\ldots, W_r$. The generalised augmentation map $\Omega : FW \to F^r$ is the $FG$-homomorphism  $\sum_{w \in W} \alpha_w w \mapsto (\ldots, \sum_{w \in W_i} \alpha_w w, \ldots)$. Denote the kernel of this by $\Aug(FW)$.

\bigskip

The main result here is:

\begin{theorem} \label{cyclic} Suppose that $M$ is  an $\omega$-categorical Ramsey structure and $W \subseteq M^{eq}$ is a finite union of $G$-orbits. Then any finitely generated submodule of $\Aug(FW)$ is cyclic (that is, generated by a single element). 
\end{theorem}

As a corollary we obtain:

\begin{corollary} \label{acccor} Suppose that $N$ is a reduct of a countable $\omega$-categorical Ramsey structure and $W \subseteq N^{eq}$ is a finite union of $G$-orbits. Suppose that $FW$ does not satisfy acc. Then there exists an infinite ascending chain of  cyclic submodules $X_1 < X_2 < \ldots $ in $FW$. 
\end{corollary}

The theorem will follow from two lemmas, which may be of independent interest. Throughout, $W\subseteq M^{eq}$ will be a finite union of $G$-orbits. We say that $G$ acts trivially on a module if it fixes every element of it.

\begin{lemma}\label{l1}  Suppose that $M$ is Ramsey and $G$ has $r$ orbits on $W$. Suppose that 
$ V_1 < V_2 < \ldots < V_t \leq FW$ are $FG$-submodules. Then 
\[ \sum\{ \dim_F(V_{i+1}/ V_i) : G \mbox{ is trivial on } V_{i+1}/V_i \mbox{  and } i < t\} \leq r.\]
\end{lemma}

\textit{Proof:\/} It suffices to prove this under the additional assumption that if $G$ is trivial on $V_{i+1}/V_i$, then $\dim_F(V_{i+1}/V_i) = 1$. By Pontryagin duality (see Section \ref{dd}) and working in the dual $\widehat{FW} = \hF^W$, we have 
\[ \Ann(V_t) < \Ann(V_{t-1}) < \ldots < \Ann(V_1) \leq \hF^W\]
and  $(\Ann(V_i)/\Ann(V_{i+1}))$ is isomorphic to the dual $\widehat{(V_{i+1}/V_i)}$. Moreover, by our additional assumption, $G$ is trivial on $V_{i+1}/V_i$ if and only if $\Ann(V_i)/\Ann(V_{i+1})$ is isomorphic to $\hF$. Note that on all of these abelian groups we have an $FG$-action.

Now,  if $\Ann(V_i)/\Ann(V_{i+1}) \cong\hF$, then we have an exact sequence $0 \to \Ann(V_{i+1}) \rightarrow \Ann(V_i) \to \hF\to 0$, where the first map is inclusion. Taking $G$-fixed elements and using Lemma \ref{exactness} we have exactness of $0\to(\Ann(V_{i+1}))^G \rightarrow (\Ann(V_i))^G \to \hF\to 0$. Taking fixed points in the above sequence of inclusions we have
\[ \Ann(V_t)^G \leq  \Ann(V_{t-1})^G \leq \ldots \leq \Ann(V_1)^G \leq (\hF^W)^G \cong \hF^r\]
(where the isomorphism at the end is from the fact that the $G$-fixed functions $W \to \hF$ are the ones which are constant on each $G$-orbit on $W$). Each of these is $F$-invariant, so when we take Pontryagin duals they are $F$-vector spaces and the inclusions become surjective maps
\[ (\Ann(V_t)^G)^{\widehat{\,}} \leftarrow  (\Ann(V_{t-1})^G)^{\widehat{\,}} \leftarrow  \ldots \leftarrow  (\Ann(V_1)^G)^{\widehat{\,}} \leftarrow   \widehat{(\hF^r)} \cong F^r . \]
The maximum number of times where these surjections are not isomorphisms is $r$ and therefore the maximum number of $i$ for which we have  $\Ann(V_i)/\Ann(V_{i+1}) \cong \hF$ is $r$. The result follows. $\Box$

\begin{lemma} \label{l2} Suppose $M$ is a Ramsey structure. Let $V = \langle v_1,\ldots, v_k \rangle_{FG}$ be a finitely generated submodule of $FW$. Then there is a cyclic submodule $X = \langle x\rangle_{FG} \leq V$ such that $G$ acts trivially on $V/X$.
\end{lemma}

\textit{Proof:\/} By Proposition 3.17 of \cite{KS}, there exist elementary submodels $M_1,\ldots, M_k \preceq M$ such that for all $i \leq k$ and finite tuples $a, b$ in $M_i$ (or $M_i^{eq}$) with $a \equiv_\emptyset b$ we have that $a, b$ have the same type over $\bigcup_{i\neq j} M_j$. (This is proved for $k = 2$ in the above reference, but of course we can then use this repeatedly to get the general case.)

By applying suitable elements of $G$, we  may assume that $v_i$ is in $M_i$ (or more precisely, that $\supp(v_i) \subseteq W(M_i)$). Let $x = v_1 + v_2 + \ldots + v_k$ and $X = \langle x\rangle_{FG}$. If $v_i' \in M_i$ and $v_i' \equiv_\emptyset v_i$, then there is $g \in G$ with $gv_i = v_i'$ and $gv_j = v_j$ for all $j \neq i$. Thus $v_i'-v_i = gx - x \in X$. If $v_i'' \in M$ and $v_i'' \equiv_\emptyset v_i$, then there is $v_i' \in M_i$ with $v_iv_i'' \equiv_\emptyset v_iv_i'$. It follows that for every $h \in G$ we have $hv_i-v_i \in X$. Thus $G$ is trivial on $V/X$. $\Box$

\begin{remarks} \rm Of course, there are structures which are not Ramsey, but where one can find  $M_i$ as used in the proof. The result also holds in these cases.
\end{remarks}

\textit{Proof of Theorem \ref{cyclic}:\/}  Let $r$ be the number of $G$-orbits on $W$.

Let $V \leq \Aug(FW)$ be a f.g. $FG$-submodule and let $X \leq V$ be a cyclic submodule such that $G$ is trivial on $V/X$, as guaranteed by Lemma \ref{l2}. Apply Lemma \ref{l1} to the sequence of modules 
\[ X \leq V \leq \Aug(FW) \leq FW.\]
As $\dim_F(FW/\Aug(FW)) = r$ and $G$ is trivial on $V/X$ we conclude that $X = V$. $\Box$

Before proving the Corollary, we note the following (which may be well-known and holds more generally).

\begin{lemma} Let $M$ be an $\omega$-categorical Ramsey structure and $G = \Aut(M)$. Suppose $W$ is a finite union of $G$-orbits on $M^{eq}$ and $U \leq V \leq FW$ are $FG$-submodules such that $V/U$ has finite $F$-dimension. Then $G$ acts trivially on $V/U$.
\end{lemma}

\textit{Proof of Lemma:\/} Let $G_1$ be the kernel of the action of $G$ on $V/U$. So $G_1$ is a closed normal subgroup of $G$. Now, the stabiliser of a vector  $v+U \in V/U$ is open in $G$ as it contains the stabiliser of the support of $v$. As $V/U$ is finite dimensional, it follows that $G_1$ is also open (it is the stabiliser of a basis). Thus $G_1$ is a clopen normal subgroup of $G$ and in particular, it is also oligomorphic on $M$. It follows that $G/G_1$ is a profinite group and so is compact. But $G$ is extremely amenable, so by considering its action on the coset space $G/G_1$, we deduce that $G_1 = G$, as required. $\Box$

\textit{Proof of Corollary \ref{acccor}:\/} Suppose that $N$ is a reduct of the $\omega$-categorical Ramsey structure $M$ and $W \subseteq N^{eq}$ is a finite union of $H$-orbits, where $H = \Aut(N)$. Let $G = \Aut(M)$. So $G \leq H$ and by $\omega$-categoricity, $W$ is also a finite union of $G$-orbits. 

Suppose $U_0 < U_1 < U_2 < \ldots \leq FW$ is an infinite ascending chain of $FH$-modules. These are also $FG$-modules, so by the above lemma and Lemma \ref{l1} we may assume that $U_{i+1}/U_i$ is infinite dimensional for all $i$. Let $Z = \Aug_{FG}(FW)$. This has finite codimension in $FW$ and so $Z\cap U_i$ is of finite codimension in $U_i$. Thus $Z\cap U_i < Z\cap U_{i+1}$ for all $i$. 

For $i \geq 1$, let  $u_i \in Z \cap U_i \setminus U_{i-1}$ and $V_i' = \langle u_1,\ldots, u_i\rangle_{FG}$. In particular, $V_i' \leq Z$, so by Theorem \ref{cyclic}, there is $x_i \in V_i'$ which generates $V_i'$. As $u_i \not\in U_{i-1}$,  we have $x_i \in U_i \setminus U_{i-1}$. Let $V_i = \langle x_i \rangle_{FH}$. So $V_i \subseteq U_i \setminus U_{i-1}$ and it follows that  $V_1 < V_2 < \ldots \leq FW$ is an infinite ascending chain of cyclic $FH$-modules. $\Box$

\bigskip\bigskip\bigskip

\renewcommand{\bibsection}{\begin{flushright}\Large
\textls[150]
{
REFERENCES}\\
\rule{8cm}{0.4pt}\\[0.8cm]
\end{flushright}}

\bigskip\bigskip


\begin{flushleft}
\rule{8cm}{0.4pt}\\
\end{flushleft}

{
\sloppy
\noindent
David M. Evans

\noindent 
Department of Mathematics

\noindent 
Imperial College London

\noindent 
London SW7~2AZ

\noindent 
UK.

\noindent
e-mail: david.evans@imperial.ac.uk
}

\bigskip
\bigskip

\hypertarget{last_page}{}\label{last_page}

\end{document}